\documentclass[preprint,table,a4paper]{article}
\usepackage[a4paper]{geometry}
\usepackage[english]{babel}
\usepackage[utf8x]{inputenc}
\usepackage{amsmath}
\usepackage{amssymb}
\usepackage{amsfonts}
\usepackage{amsmath}
\usepackage{amsthm}
\usepackage{graphicx}
\usepackage{pgfplots}
\usepackage{tabularx}
\usepackage{xcolor}
\usepackage{enumerate}
\usepackage{url,bm,comment}

\DeclareGraphicsExtensions{.pdf,.eps,.png,.jpg}
\usepackage{epstopdf}

\newcommand{\cu}{{\bm i}}
\newcommand{\norm}[1]{\lVert #1 \rVert}

\definecolor{Gray}{gray}{0.95}

\newtheorem{thm}{Theorem}[section]

\newtheorem{lem}[thm]{Lemma}

\newtheorem{rem}[thm]{Remark}
\newtheorem{cor}[thm]{Corollary}
\newtheorem{property}[thm]{Property}

\renewcommand{\leq}{\leqslant}

\renewcommand{\le}{\leqslant}
\renewcommand{\ge}{\geqslant}

\title{Efficient cyclic reduction for QBDs with 
 rank structured blocks\thanks{Work supported by Gruppo Nazionale di Calcolo Scientifico (GNCS) of INdAM, and by the PRA project ``Mathematical models and computational methods for complex networks'' funded by the University of Pisa.}}
 
\author{Dario A. Bini ({\tt bini@dm.unipi.it}),\\
 Dipartimento di Matematica, University of Pisa, Italy
\\
Stefano Massei ({\tt stefano.massei@sns.it}), \\
Scuola Normale Superiore, Pisa, Italy \\
Leonardo Robol  ({\tt leonardo.robol@sns.it}),\\
Scuola Normale Superiore, Pisa, Italy}
\begin{document}\maketitle

  \begin{abstract}
We provide effective algorithms for solving block tridiagonal block Toeplitz systems with $m\times m$ quasiseparable blocks, as well as quadratic matrix equations with $m\times m$ quasiseparable coefficients, based on cyclic reduction and on the technology of rank-structured matrices. The algorithms rely on the exponential decay of the singular values of the off-diagonal submatrices generated by cyclic reduction. We provide a formal proof of this decay in the Markovian framework.    
The results of the numerical experiments that we report confirm a significant speed up over the general algorithms, already starting with the moderately small size $m\approx 10^2$.
  \end{abstract}

  \section{Introduction}
Cyclic reduction (CR)
is an effective tool that can be used for
solving several problems in linear algebra and in polynomial computations \cite{CR}. It has been originally introduced by G.H.~Golub and
R.W.~Hockney in the mid 1960s \cite{hockney,buzbee},  for the numerical solution of block tridiagonal
linear systems stemming from the finite differences solution of
elliptic problems, and has been generalized to solve nonlinear matrix
equations associated with matrix power series with applications to queuing problems, Markov chains and spectral decomposition of polynomials. In fact, an important
application of CR concerns the computation of the minimal nonnegative
solution of the matrix equation $X=A_{-1}+A_0 X+A_1X^2$, encountered in
Quasi Birth-Death (QBD) Markov chains, where $A_{-1}$, $A_0$, and
$A_1$ are given $m\times m$ nonnegative matrices such that
$A_{-1}+A_0+A_1$ is irreducible and stochastic and where $X$ is the matrix unknown \cite{SMC,CR}.
The computation of the solution $X$ allows to recover
the steady state vector $\pi$ of the Markov
chain. 

Rewriting the matrix equation as $A_{-1}+(A_0-I)X+A_1X^2=0$, 
CR computes four sequences of matrices, $A_i^{(h)}$, $i=-1,0,1$ and $\widehat A_0^{(h)}$, according to the following equations
\begin{equation}\label{eq:cr}
  \begin{split}
  A_1^{(h+1)} &= -A_1^{(h)}\ S^{(h)}\ A_1^{(h)},\qquad   S^{(h)} =(A_0^{(h)}-I)^{-1}\\
  A_0^{(h+1)} &= A_0^{(h)} - A_{1}^{(h)}\ S^{(h)}\ A_{-1}^{(h)}\ -\ A_{-1}^{(h)}\ S^{(h)}\ A_{1}^{(h)},\\
  A_{-1}^{(h+1)} &= -A_{-1}^{(h)}\ S^{(h)}\ A_{-1}^{(h)},\quad
  \widehat A_0^{(h+1)} = \widehat A_0^{(h)} - A_{1}^{(h)}\ S^{(h)}\ A_{-1}^{(h)},
  \end{split}
  \end{equation}
for $h=0,1,\ldots$,  with $A_i^{(0)} = A_i$, $i=-1,0,1$ and $\widehat A_0^{(0)}=A_0-I$. It can be proved,
in the context of Markov chains, \cite{SMC} that
the sequence $-(\widehat A_0^{(h)})^{-1}A_{-1}$ converges to the minimal nonnegative solution $G$ of the matrix equation. 
 
  If the QBD process is not null recurrent, applicability and quadratic convergence of the algorithm are guaranteed \cite{SMC}[Theorems 7.5, 7.6]. In the null recurrent case, it has been proved in \cite{guo2002convergence} that convergence is linear with factor $\frac 1 2$.

  Without any further assumption on the structure of the blocks, each
  step of CR requires a small number of matrix multiplications 
  and one matrix inversion for the resulting computational cost of $O(m^3)$
  arithmetic operations (ops) per step.  On the other hand, there are several
  models from the applications in which the blocks $A_i$ exhibit
  special structures.  In order to decrease the computational
  complexity of the iterations, variations of CR
  which exploit these structures have been proposed, see for instance \cite{BTCR}, \cite{PerVan}, \cite{LRCR}.

Here, we are interested in analysing the case where the blocks $A_i$
are quasiseparable matrices. That is, the case where the {\em off-diagonal submatrices} of $A_{-1},~A_0$ and $A_1$, strictly contained in the upper or in the lower triangular part, have low rank with respect to $m$. The maximum of the ranks of the off-diagonal submatrices is called {\em quasiseparable rank} and a matrix with quasiseparable rank $k$ is called $k$-quasiseparable. Observe that $k$-quasiseparable matrices include banded matrices. These structures are
encountered in wide and important classes of applications like, 
 bidimensional random walks
\cite{Miya},  the Jackson tandem queue model \cite{tandem} and other QBD processes, or, for instance,
in the finite differences discretization of 
elliptic PDEs.

Our goal is to design a version of CR which exploits the rank structures of
the blocks $A_i$ and which can be implemented at a substantially lower
cost. This way, we may arrive at designing effective solvers both for block
tridiagonal block Toeplitz systems and for the quadratic matrix
equations encountered in QBD Markov chains.
Indeed, a way to reach this goal is to find out  if some structure of the blocks $A_i^{(h)}$ is maintained during the CR steps \eqref{eq:cr}, and then to try to exploit this structure in order to design an effective implementation of CR.

Looking at the iterative scheme \eqref{eq:cr},
one can find out that the quasiseparable rank can grow exponentially at each step. Despite that, plotting the 
singular values of the off-diagonal blocks of the
matrices $A_i^{(h)}$ shows an interesting
behaviour as reported in Figure \ref{fig:decay}.
For a randomly generated QBD process with tridiagonal blocks of size $1600$, the singular values of an off-diagonal submatrix of size $799\times 800$ in $A_0^{(h)}$ have an exponential decay in the first 20 steps of CR.

   \begin{figure}[!ht]\label{fig:decay}
\centering
\includegraphics[height=0.3\textheight]{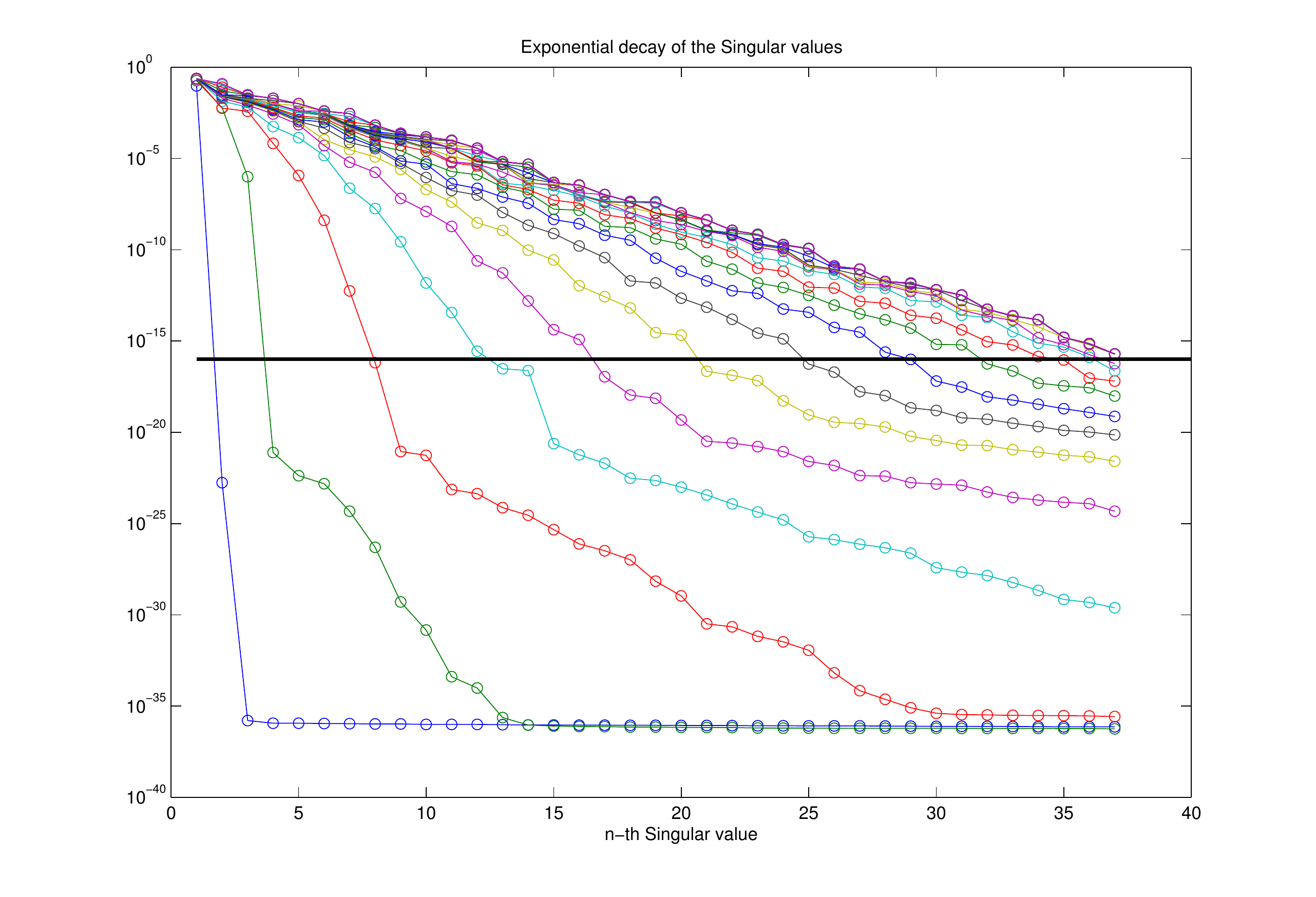}
\caption{Log-scale plot of the largest singular values of the largest south-western  submatrix of $A_0$ contained in the lower triangular part, for $m=1600$. The horizontal line denotes the machine precision threshold. Matrices are randomly generated so that
$A_i\ge 0$ are tridiagonal matrices and $A_{-1}+A_0+I+A_1$ is stochastic.}
\end{figure}

It is evident that, even though the number of nonzero singular values grows at each step of CR, the number of singular values above the machine precision -- denoted by a horizontal line in Figure \ref{fig:decay}-- is bounded by a moderate constant.
Moreover, the singular values seem to stay below a straight-line which constitutes an asymptotic bound. That is, they get closer to this line as $h\to\infty$. The logarithm scale suggest that the computed singular values $\sigma^{(h)}_i$ decay exponentially with $i$ and the basis of the exponential grows with $h$ but has a limit less than 1.

We will prove this asymptotic property relying on the technology of rank-structured matrices and relate the basis of the exponential decay to the width of the domain of analiticity of the matrix function $\varphi(z)^{-1}$ for
$
\varphi(z)=-z^{-1}A_{-1}+I-A_0-z A_1
$.

The paper is organized as follows. In Section~\ref{sec:main-properties} we 
recall some useful properties of CR. 
In Section \ref{sec:technical} we prove some preliminary lemmas for bounding the singular values of sums,
products and inverses of matrices. 
Section \ref{sec:exponential-decay} contains the proof of the main result concerning the exponential
decay of the singular values. 
In Section \ref{sec:hmatrices} we describe our algorithm which implements CR relying on the package H2Lib \cite{h2lib}, concerning $\mathcal H$-matrices, and the related code together with the results of the numerical experiments. Finally, 
in Section~\ref{sec:conclusions}, we give 
some concluding comments and possible future developments. 

\section{Main properties of cyclic reduction}
  \label{sec:main-properties}
We recall a functional interpretation of cyclic reduction introduced in the Markov chains framework \cite{SMC}, \cite{CR} to prove applicability and convergence properties.

Associate the matrices $A^{(h)}_i$, $i=-1,0,1$ defined in \eqref{eq:cr} with  
the matrix Laurent polynomial
$
\varphi^{(h)}(z):=-z^{-1}A_{-1}^{(h)}+ (I-A_0^{(h)})-zA_1^{(h)}
$,
where $\varphi^{(0)}(z)=\varphi(z)=z^{-1}A_{-1}+(I-A_0)-zA_1$,
and define
the matrix rational function 
$\psi^{(h)}(z)=\varphi^{(h)}(z)^{-1}$. The following property holds
\[
\begin{cases}
\psi^{(0)}(z):=\psi(z),\\
\psi^{(h+1)}(z^2):=\frac{1}{2}(\psi^{(h)}(z)+\psi^{(h)}(-z)).
\end{cases}
\]

In particular, making explicit the recurrence relation in the sequence $\{\psi^{(h)}\}$, we find that 
$\psi^{(h)}(z^{2^h})= \frac 1{2^h}\sum_{j=0}^{2^h-1}\psi^{(0)}(\omega_N^j z)
$
where $\omega_N=e^{\frac{2\pi}{N}\cu}$ is a principal $N$-th root of the unity for $N=2^h$, and $\cu$ denotes the imaginary unit,
so that
\begin{equation}\label{radix}
\varphi^{(h)}(z^{2^h})=   \left(
  \frac{1}{2^h}\sum_{j=0}^{2^h-1}\psi^{(0)}(\omega_N^j z)
  \right)^{-1}.
\end{equation}

Observe that in the case where $A_{-1}$, $A_0$ and $A_1$ are tridiagonal, then $\varphi(z)$ is tridiagonal as well, so that for any value of $z$ such that $\det\varphi(z)\ne 0$, the matrix $\psi(z)$ is semiseparable, that is, $\hbox{tril}(\psi(z))=\hbox{tril}(L)$, $\hbox{triu}(\psi(z))=\hbox{triu}(U)$, where $L$ and $U$ are matrices of rank 1 \cite{vanbarel:book1}.

Relying on this property, in the next section we will show that $\psi_h(z)$, as well as the blocks $A_{-1}^{(h)}$, $A_0^{(h)}$ and $A_1^{(h)}$, has off-diagonal submatrices with singular values which decay exponentially.  

  \section{Singular values of sums, products and inverses}
    \label{sec:technical}
    In this section, we recall some basic facts on the singular values decomposition (SVD) and provide some technical lemmas.
Let us denote by $\sigma_j(A)$
 the $j$-th singular value of the $m\times n$ matrix $A$ sorted in 
 non-ascending order and by $A=U\Sigma V^*$, $\Sigma=\hbox{diag}(\sigma_1,\ldots,\sigma_p)$, $p=\min\{m,n\}$,
     the SVD of $A$  where $U$ and $V$ are unitary and $V^*$ denotes the transpose conjugate of $V$. 
Moreover, denote by $u_i$ and $v_i$ the $i$th column of $U$ and $V$, respectively.
We recall the following well-known property concerning the SVD of a matrix $A$.

\begin{property}\label{prop} For any matrix $A$
the function $f(X)=\|A-X\|_2$ takes its minimum over the class of matrices of rank $\ell-1$ for $X=\sum_{i=1}^{\ell-1} \sigma_iu_iv_i^*$ and the value of the minimum is exactly $\sigma_{\ell}$.
\end{property}

\subsection{Some technical lemmas}    
Let $\| A\|_2$ be the Euclidean norm of $A$ and $\mu_2(A)=\|A\|_2\cdot\|A^{-1}\|_2$ be the condition number of $A$.   The following result relates the singular values of the matrices $A$, $B$ and $C=AB$.
    
    \begin{lem} \label{lem:dario1}
      Consider two matrices $A \in \mathbb{C}^{m \times n}$, 
      $B\in\mathbb{C}^{n \times n}$, such that
      $B$ is invertible. Then
      \[\begin{split}
       & \frac{\sigma_j(A)}{\lVert B^{-1} \rVert_2} \leq \sigma_j(AB) \leq \lVert B \rVert_2 \cdot \sigma_j (A), \qquad
        \frac{\sigma_j(A)}{\lVert B^{-1} \rVert_2} \leq \sigma_j(BA^*) \leq \lVert B \rVert_2 \cdot \sigma_j (A), \\
 \end{split}
      \]
    \end{lem}
    
    \begin{proof}
      We prove only the first statement since the other follows directly from it.
      Consider the SVD decompositions 
      $
        A = U_A \Sigma_A V_A^* $,  
        $B = U_B \Sigma_B V_B^*. 
      $
      Recall that the singular values of a generic matrix 
      $M \in \mathbb{C}^{m \times n}$
      are the square roots of the eigenvalues of $M^* M$, so that by the minimax theorem, 
      we can write 
      \[
        \sigma_j(M)^2 = \max_{\substack{\dim(V) = j \\ V \subseteq \mathbb{C}^{n}}}\, 
\min_{\substack{x \in V\\ x\ne 0}} \frac{x^* M^* M x}{x^* x}. 
      \]
      Now note that $AB = U_A\Sigma_AV_A^* U_B \Sigma_B V^*$ 
      and since unitary matrices do not change the singular
      values we have $\sigma_j(AB) = \sigma_j(\Sigma_A Q \Sigma_B)$ where $Q = V_A^* U_B$. 
      Thus, we can express $\sigma_j(AB)^2$ as 
      \[
        \sigma_j(AB)^2 = \max_{\substack{\dim(V) = j \\ V \subseteq \mathbb{C}^{n}}} \min_{x \in V} \frac{x^* \Sigma_B^* Q^* \Sigma_A^* \Sigma_A Q \Sigma_B x}{x^* x} = 
        \max_{\substack{\dim(V) = j \\ V \subseteq \mathbb{C}^{n}}} \min_{x \in V} 
          \frac{ (\Sigma_B x)^* Q^* \Sigma_A^2 Q (\Sigma_B x)}{x^* x}. 
      \]
      By setting $y = \Sigma_B x$ and recalling that
      $\Sigma_B$ is invertible by hypothesis we have
      \[
        \sigma_j(AB)^2 = \max_{\substack{\dim(V) = j \\ V \subseteq \mathbb{C}^{n}}} \min_{y \in V} 
          \frac{y^* Q^* \Sigma_A^2 Q y}{y^* y}\cdot 
          \frac{y^* y}{x^* x}
      \]
      so that, by using the fact that $Q$ is unitary 
      and that $\frac{x^* x}{\lVert B^{-1} \rVert_2^2} \leq y^* y
        \leq \lVert B \rVert_2^2 x^* x$,
      we obtain
      $
        \frac{\sigma_j(A)^2}{\lVert B^{-1} \rVert_2^2} = \frac{\sigma_j(\Sigma_A)^2}{\lVert B^{-1} \rVert_2^2} \leq \sigma_j(AB)^2 \leq \lVert B \rVert_2^2 \cdot \sigma_j(\Sigma_A)^2 = 
        \lVert B \rVert_2^2 \cdot \sigma_j(A)^2. 
      $
    \end{proof}
    
    The following lemma relates the singular values of an off-diagonal submatrix of the inverse of a given matrix $A$ with those of the corresponding submatrix of $A$.
    
    \begin{lem}\label{lem:dario2}
    Let $\mathcal{A}\in\mathbb{C}^{n\times n}$ be an invertible matrix and  consider the block partitioning
    \[
    \mathcal{A}=\left(\begin{array}{cc}
    A&B\\
    C&D
    \end{array}\right),\ \ 
    \mathcal{A}^{-1}=\left(\begin{array}{cc}
    \tilde{A}&\tilde{B}\\
    \tilde{C}&\tilde{D}
    \end{array}\right),
    \]
    where $A$ and $\tilde A$ are $i \times i$
    matrices for $2\le i\le n-1$. We have the following properties
    \begin{enumerate}
    \item \label{en:first} If $D$ is invertible then 
    $\frac{1}{\lVert D \rVert_2 \lVert S_D \rVert_2} \sigma_j(C) \leq 
         \sigma_j(\tilde{C})
      \leq  
         \lVert D^{-1} \rVert_2 \cdot \lVert S_D^{-1} \rVert_2\cdot \sigma_j({C})
    $,
where $S_D=A-BD^{-1}C$ is the Schur complement of $D$.
    \item \label{en:second} If $A$ is invertible then
$
      \frac{1}{\lVert A \rVert_2 \lVert S_A \rVert_2} \sigma_j(C) \leq 
         \sigma_j(\tilde{C})
      \leq  
         \lVert A^{-1} \rVert_2 \cdot \lVert S_A^{-1} \rVert_2\cdot \sigma_j({C})
 $,
         where $S_A=D-CA^{-1}B$ is the Schur complement of $A$.
    \end{enumerate}
    \end{lem}
    
    \begin{proof}
    Let us consider part~\ref{en:first}. 
    If $D$ is invertible we can write
    \[
    \left(\begin{array}{cc}
    A&B\\
    C&D
    \end{array}\right)^{-1}=
    \left(\begin{array}{cc}
      S_D^{-1}&S_D^{-1}BD^{-1}\\
    -D^{-1}CS_D^{-1}&D^{-1}+D^{-1}CS_D^{-1}BD^{-1}
    \end{array}\right)
    \]
    and in particular we have $\tilde{C}=-D^{-1}CS_D^{-1}$.  
    Repeatedly applying Lemma~\ref{lem:dario1} to $\tilde C$
    yields
    $
      \frac{1}{\norm{D}_2} \sigma_j(CS_D^{-1}) \leq 
        \sigma_j(\tilde C) \leq
        \norm{D^{-1}}_2 \cdot \sigma_j(CS_D^{-1})
    $
    and eventually
    $
      \frac{1}{\norm{S_D}_2} \sigma_j(C) \leq
        \sigma_j(CS_D^{-1}) \leq 
        \norm{S_D^{-1}}_2 \cdot \sigma_j(C). 
   $
    Combining these inequalities gives us the thesis. 
    For proving part~\ref{en:second},  we  can proceed  in the same manner relying on the inversion formula
\[
\left(\begin{array}{cc}
A&B\\
C&D
\end{array}\right)^{-1}=
\left(\begin{array}{cc}
A^{-1}+A^{-1}BS_A^{-1}CA^{-1}&-A^{-1}BS_A^{-1}\\
-S_A^{-1}CA^{-1}&S_A^{-1}
\end{array}\right).
\] 
    \end{proof}

    \begin{lem}\label{dyads}
    Let $A=\sum_{i=-\infty}^{+\infty} A_i$ 
    and $A^+=\sum_{i=0}^{+\infty} A_i$  
    where $A_i\in\mathbb{C}^{m\times n}$ have rank at most $k$ and suppose that $\left\| A_i\right\|_2\leq Me^{-\alpha |i|}$. Then 
    $
    \sigma_j(A)\leq \frac{2M}{1-e^{-\alpha}}\cdot e^{-\alpha  \frac{j-k}{2k}}$,
   $ \sigma_j(A^+)\leq \frac{M}{1-e^{-\alpha}}\cdot e^{-\alpha  \frac{j -k}{k}}
    $.
    \end{lem}
    \begin{proof}
Note that $B_s=\sum_{i=-s+1}^{s-1}A_i$ has rank at most $k(2s-1)$. For
$j$ positive integer, set $s=\lceil \frac{j-k}{2k}\rceil$ and observe
that, since $2k\lceil\frac{j-k}{2k}\rceil< j+k$, we have $k(2s-1)\le
j-1$ so that $B_s$ is an approximation to $A$ of rank at most
$j-1$. By Property \ref{prop} it follows that
$\sigma_{j}(A)\le\|A-B_s\|_2$. Moreover, since $A-B_s=\sum_{|i|\ge
  s}A_i$, we have $\sigma_{j}(A)\le \sum_{|i|\ge s}\|A_i\|_2\le 2M
e^{-\alpha s}/(1-e^{-\alpha})$ which completes the proof of the first
bound.  The second bound is proved similarly.
    \end{proof}

    \begin{rem}\label{rem:bound2}
    In the particular case where $k=1$, Lemma \ref{dyads} yields 
    $
     \sigma_s(A)\leq \frac{2M}{1-e^{-\alpha}}\cdot e^{-\alpha  \frac{s}{2}}$, $\sigma_s(A^+)\leq \frac{M}{1-e^{-\alpha}}\cdot e^{-\alpha (s-1)}.
    $
    \end{rem}

  \section{Exponential decay of the singular values}
    \label{sec:exponential-decay}
  We can now state the main result about the decay of
  the singular values. 
  It is clear that, if the blocks $A_i$ $i=-1,0,1$ have an off-diagonal    rank structure, then the matrix $\varphi^{(0)}(z)$ also enjoys this property.
  We will show that this fact implies the exponential decay of the singular values of the off-diagonal blocks of $\varphi^{(h)}(z^{2^h})$ for every $h$ and for any $z\in\mathbb T$ where  $\mathbb T=\{z\in\mathbb C:\quad |z|=1\}$ denotes the unit circle in the complex plane.
  
Given  an integer $N>0$, let $\omega_N=e^{\cu2\pi/N}$ and observe that the quantity
$\frac 1N \sum_{j=0}^{N-1} (z\omega_N^{j})^k$ coincides with $z^k$ if $k\equiv 0 \mod N$ and with $0$ otherwise.
This way, if $A(z)=\sum_{i\in\mathbb Z}z^iA_i $ is a matrix Laurent series analytic on the annulus $\mathbb A(r_1,r_2)=\{z\in\mathbb C: r_1<|z|<r_2\}$ for $0<r_1<1<r_2$, then $\frac 1N \sum_{j=0}^{N-1}A(\omega_n^jz)=B(z^N)$ where $B(z)=\sum_{i\in\mathbb Z} z^i A_{Ni}$ is analytic on  $\mathbb A(r_1^N,r_2^N)$. We denote by $I_N$ the operator which maps $A(z)$ into $B(z)$ and write $B(z)=I_N(A(z))$. Observe that $I_N$ is linear and continuous on the space of analytic functions
on $\mathbb A(r_1,r_2)$. 
  
  Moreover, in view of \eqref{radix}, we have $\psi^{(h)}=I_{N}(\psi^{(0)})$ for $N=2^h$.
  This way, if we prove that any off-diagonal submatrix $B(z)$ of $\psi^{(0)}(z)$ is such that $I_N(B(z))$ has the exponential decay property for its singular values, then we have shown this property also for $\psi^{(h)}(z)$.
  
Partition $\varphi(z)$ and $\varphi(z)^{-1}$ as follows
\begin{equation}\label{eq:partition}
  \varphi(z) = \begin{bmatrix}
    I-E(z) & -B(z) \\
    -C(z) & I-D(z) \\
  \end{bmatrix}, \qquad 
  \psi(z):=\varphi(z)^{-1} = \begin{bmatrix}
    \tilde E(z) & \tilde B(z) \\
    \tilde C(z) & \tilde D(z) \\
  \end{bmatrix}, 
\end{equation}
where $E(z)$ and $D(z)$ are square matrices of any compatible size.

  \begin{thm} \label{thm:decay}
    Let  $\varphi(z)= -z^{-1} A_{-1}+I-A_0-zA_1$ 
    be an $m\times m$  matrix    
    function such that
    \begin{enumerate}[(i)]
      \item The $A_i$ are non-negative and
          $I-\varphi(z)$ has spectral radius smaller than $1$ for any $z \in \mathbb T$. 
      \item The blocks $A_i$ are $k$-quasiseparable and $\norm{A_i}_2\leq L$, $i=-1,0,1$.
      \item There exist $t>1$ and $\delta>0$ such that $\det\varphi(z)\ne 0$ and  $\|\varphi(z)^{-1}\|_2\le \delta$ for $z\in\mathbb A(t^{-1},t)$.
    \end{enumerate} 
  Then $\rho(I-\varphi(z))<1$ for any $z\in\mathbb A$ and in the partitioning \eqref{eq:partition},  both blocks $I-E(z)$ and $I-D(z)$ are invertible for any $z\in\mathbb A$. Moreover,
 for any $z\in \mathbb T$ and for any $h$,  the singular values of $\tilde C^{(h)}:= I_N(\tilde C(z))$, with $N=2^h$, are such that
    \begin{equation}\label{eq:thm:decay}
    \sigma_{s}(\tilde C^{(h)}(z)) \leq 3M e^{-\frac{s-3k}{6k}\log t},
\qquad M=\frac{4L\delta^2}{(1-e^{-N\log t})(1-t^{-1})}.
    \end{equation}
    Moreover, if $A_{-1}$, $A_0$, $A_1$ are tridiagonal then the above bound turns into
    \begin{equation}\label{eq:bound2}
    \sigma_{s}(\tilde C^{(h)}(z)) \leq    
M e^{-\frac s2\log t}.
    \end{equation}
  \end{thm}
  
  \begin{proof}
  Let us prove that $\rho(I-\varphi(z))<1$ for any $z\in\mathbb A$. By contradiction,
  assume that there exists $\xi\in\mathbb A$ such that $\rho(I-\varphi(\xi))\ge 1$. Since $A_i\ge 0$ for $i=-1,0,1$  then
  $|I-\varphi(\xi)|\le I-\varphi(|\xi|)$, and
 by the monotonicity of the spectral  radius we get $1\le \rho(I-\varphi(\xi))\le\rho(I-\varphi(|\xi|))$. Thus, since
   $\rho(I-\varphi(1))<1\le\rho(I-\varphi(|\xi|))$ and $\rho$ is a continuous function, then there exists $1/t< \hat\xi<t$ such that $\rho(I-\varphi(\hat\xi))=1$. Since $I-\varphi(\hat\xi)$ is nonnegative, then by the Perron-Frobenius theorem there exists an eigenvalue of  $I-\varphi(\hat\xi)$ equal to 1, that is $\varphi(\hat\xi)$ would be singular, which contradicts the assumptions. 
   
 Now we prove that $I-D(z)$ and $I-E(z)$ are invertible for any $z\in\mathbb A$. 
Since $|D(z)|\le D(|z|)$,  for the monotonicity of the spectral radius, we have $\rho(D(z))\le\rho(|D(z)|)\le\rho(D(|z|))$. On the other hand, $D(|z|)$ is a principal submatrix of the nonnegative matrix $I-\varphi(|z|)$ so that $\rho(D(|z|))\le\rho(I-\varphi(|z|))$ 
which is less than 1 since $|z|\in\mathbb A$. 
  We conclude that $\rho(D(z))<1$ for any $z\in\mathbb A$ so that $I-D(z)$ is nonsingular. The same argument can be used to deduce that $I-E(z)$ is nonsingular.
  
Now we prove the bound \eqref{eq:thm:decay} on the singular values.
For simplicity we assume that $k=1$, the general case can be treated
similarly.  Since the off-diagonal blocks of $A_i$ have rank at most
$1$ then $C_i=u_iv_i^T$, $i=-1,0,1$, for suitable vectors $u_i,v_i$
where we assume that $\|u_i\|_2=\|C_i\|_2$, $\|v_i\|_2=1$. Thus, we
have $C(z)=\sum_{i=-1}^1z^iu_iv_i^T$. Since $I-D(z)$ is invertible on
$\mathbb A$, we have $\tilde C(z)=H(z)\sum_{i=-1}^1z^iu_iv_i^T K(z)$
where $H(z)=(I-D(z))^{-1}$, $K(z)=S_D(z)^{-1}=\tilde E(z)$, and
$H(z)$, $K(z)$ are analytic for $z\in\mathbb A$. Consider the Fourier
series of $H(z)$ and $K(z)$, that is, $H(z)=\sum_{s\in\mathbb
  Z}z^sH_s$, $K(z)=\sum_{s\in\mathbb Z}z^sK_s$, and recall that the
coefficients $H_s$, $K_s$ have an exponential decay
\cite{henrici}[Theorem 4.4c], that is, $
|(H_s)_{i,j}|\le\max_{z\in\mathbb A}|(H(z))_{i,j}|e^{-|s|\log t}$,
$|(K_s)_{i,j}|\le\max_{z\in\mathbb A}|(K(z))_{i,j}|e^{-|s|\log t}$.
Since for any matrix norm induced by an absolute norm $\|\cdot\|$ and
for any matrix $A$ it holds that $|a_{i,j}|\le\|A\|$ so that we may
write
 \begin{equation}\label{eq:decay}
 \|H_s\|\le \max_{z\in\mathbb A} \|H(z)\|e^{-|s|\log t},\quad
 \|K_s\|\le \max_{z\in\mathbb A} \|K(z)\|e^{-|s|\log t},\quad
 \end{equation}
 Now recall that $\tilde C=H(z)\sum_{i=-1,0,1}z^iu_iv_i^TK(z)$, set $z\in\mathbb T$ and consider the generic $i$th term
 $z^iH(z)u_iv_i^TK(z)$
  in the above summation. We have
 \[
z^i H(z)u_iv_i^TK(z)= \sum_{s,h\in\mathbb Z}z^{s+h+i}H_su_{i}v_{i}K_h=
 \sum_{s\in\mathbb Z}H_su_i\sum_{p\in\mathbb Z}z^{p+i}v_i^TK_{p-s},
 \]
 where we have set $p=s+h$. Now, applying the operator $I_N$ to the above matrix cancels out the terms in $z^{p+i}$ such that $p+i$ is not multiple of $N$, so that we are left with the terms where $p+i=Nq$ and we get
 \[
 I_N(z^iH(z)u_iv_i^TK(z))=\sum_{s\in\mathbb Z}H_su_i\sum_{q\in\mathbb Z}z^qv_i^TK_{Nq-i-s}
 =:\sum_{s\in\mathbb Z}\hat u_s^{(i)}\hat v_s^{(i)}(z),
 \]
 for $\hat u_s^{(i)}=H_su_i$, $\hat v_s^{(i)}(z)=\sum_{q\in\mathbb Z}z^qv_i^TK_{Nq-i-s}$.
Thus we may write
\[
I_N(\tilde C)=\sum_{s\in\mathbb Z}\hat U_s\hat V_s(z)^T,\quad \hat U_s=\left[\hat u_s^{(-1)},\hat u_s^{(0)},\hat u_s^{(1)}\right],\quad
\hat V_s(z)=\left[\hat v_s^{(-1)}(z),\hat v_s^{(0)}(z),\hat v_s^{(1)}(z)\right].
\]
To complete the proof, recall that $z\in\mathbb T$ and apply Lemma \ref{dyads} with $k=3$ to the series $\sum_{s\in\mathbb Z}\hat U_sV_s^T$. In order to do this, we have to provide upper bounds to  $\|\hat U_s\hat V_s(z)^T\|_2$ for $z\in\mathbb T$. We have $\|\hat U_s\hat V_s(z)^T\|_2\le\|\hat U_s\|_2\|\hat V_s(z)\|_2$. Concerning $\|\hat U_s\|_2$, since $\hat U_s=H_s\left[u_{-1},u_0,u_1\right]$, we have
$\|\hat U_s\|_2\le\|H_s\|_2\|\left[u_{-1},u_0,u_1\right]\|_2\le \sqrt 3\|H_s\|_2
\max_i\|C_i\|_2$, where the latter inequality follows from the fact that $\|u_i\|_2=\|C_i\|_2$ and that consequently, $\|\left[u_{-1},u_0,u_1\right]\|_2\le \sqrt 3
\max_i\|C_i\|_2$.
Thus from \eqref{eq:decay} we get
\[
\|\hat U_s\|_2\le\sqrt 3L \max_{z\in\mathbb A}\|H(z)^{-1}\|_2 e^{-|s|\log t}.
\]
Similarly, since $\|v_i\|_2=1$ and $|z|=1$, we have 
\[
\|\hat v_s^{(i)}\|_2\le \sum_{q\in\bm Z}\|K_{Nq-i-s}\|_2\le \max_{z\in\mathbb A}\|K(z)\|_2
\sum_{q\in\bm Z} e^{-|Nq-i-s|\log t},
\] where the last inequality follows from
 \eqref{eq:decay}. Define $r$ the remainder of the division of $i+s$ by $N$, so that $i+s=N\hat q+r$, and get  
$\sum_{q\in\bm Z} e^{-|Nq-i-s|\log t}=  \sum_{q\in\bm Z} e^{-|N(q-\hat q)+r|\log t}=
\sum_{q\in\bm Z} e^{-|Nq+r|\log t}=
e^{-r\log t}+\sum_{q\ge 1}e^{-(Nq-r)\log t}+\sum_{q\ge 1}e^{-(Nq+r)\log t}=
e^{-r\log t}+(e^{r\log t}+e^{-r\log t})(\frac 1{1-e^{-N\log t}}-1)\le \frac 2{1-e^{-N\log t}}$.
Whence we deduce that
\[
\|\hat V_s\|_2\le \frac {2\sqrt 3}{1-e^{-N\log t}}\max_{z\in\mathbb A}\|K(z)\|_2.
\]
Combining the two bounds yields
\begin{equation}\label{eq:bt}
\|\hat U_s\hat V_s(z)\|_2\le\frac{6L}{1-e^{-N\log t}}
\max_{z\in\mathbb A}\|K(z)\|_2\max_{z\in\mathbb A}\|H(z)\|_2e^{-|s|\log t}.
\end{equation}
It remains to estimate $\|K(z)\|_2$ and $\|H(z)\|_2$. Concerning $K(z)=\tilde E(z)$, observe that this is a principal submatrix of $\psi(z)$ so that $\|K(z)\|_2\le\|\psi(z)\|_2$. Concerning $H(z)=(I-D(z))^{-1}$, observe that
from the condition $A_i\ge 0$ it follows that $|D(z)|\le D(|z|)$ and that $\rho(D(z))\le
\rho(|D(z)|)\le\rho(D(|z|))\le\rho(\varphi(|z|))<1$ since $I-D(z)$ is a principal submatrix of $\varphi(z)$. Thus we may write
$(I-D(z))^{-1}=\sum_{j=0}^\infty D(z)^j$ and 
$|(I-D(z))^{-1}|\le (I-D(|z|))^{-1}$.
Now, since $A_i\ge 0$ for $i=-1,0,1$, then
\[
\tilde D(|z|)=(I-D(|z|))^{-1}+\underbrace{(I-D(|z|))^{-1}}_{\ge 0}\underbrace{C(|z|)}_{\le 0}
\underbrace{S^{-1}_{I-D(|z|)}}_{\ge 0}\underbrace{B(|z|)}_{\le 0}\underbrace{(I-D(|z|))^{-1}}_{\ge 0}
\ge (I-D(|z|))^{-1}
\]
 so that $\|(I-D(z))^{-1}\|_2
\le \|(I-D(|z|))^{-1}\|_2\le \|\tilde D(|z|)\|_2\le \max_{z\in\mathbb A}\|\psi(z)\|_2$. 
Thus, applying Lemma \ref{dyads} together with the bound \eqref{eq:bt} and rank of the blocks 3 yields
\[
\sigma_s(\tilde C^{(h)}(z))\le \frac{12L\delta^2}{(1-e^{-N\log t})(1-t^{-1})}e^{-\frac{s-3}6\log t}.
\]  
If the blocks $A_i$ are $k$-quasiseparable, then Lemma \ref{dyads} is applied with rank of the blocks $3k$ so that the exponent $(s-3)/6$ is replaced by $(s-3k)/(6k)$,
 If $\varphi(z)$ is tridiagonal, then $u_{-1}=u_0=u_1$ and $v_{-1}=v_0=v_1$, so that $\hat U_j$ and $\hat V_j$ are formed by a single column, i.e., Lemma \ref{dyads} is applied with rank of the blocks 1. This provides \eqref{eq:bound2}. 
\end{proof}

  \subsection{Exponential decay of the singular values in $\varphi(z)$}
  In this section, we prove the decay property of the singular values in the off-diagonal submatrices of $\varphi^{(h)}(z)$ when $|z|=1$. The proof is obtained by combining the decay property for the matrix function $\psi^{(h)}$, stated in Theorem~\ref{thm:decay}, with a suitable lemma which allows to extend this property to the matrix inverse.

    \begin{lem} \label{lem:well-conditioning}
   Let $\varphi^{(h)}(z) = -z^{-1} A^{(h)}_{-1} + I - A^{(h)}_0 - z A^{(h)}_1$ be the 
   $m\times m$-matrix Laurent polynomial obtained at the $h$th step of CR. Under the hypotheses of Theorem~\ref{thm:decay}, 
  for every $z\in \mathbb T$ we have the following bound: 
  \[
     \sigma_j(C^{(h)}) \leq K(L_h,\varphi) \cdot \sigma_j(\tilde C^{(h)}),\qquad K(L_h,\varphi)=(1+3L_h)(1+L_h+L_h^2\norm{\varphi^{}(1)^{-1}}_2)
    \] 
    where   $\varphi^{(h)}(z)$ and $\varphi^{(h)}(z)^{-1}$ are partitioned as in \eqref{eq:partition} and $L_h$ is such that $\|A^{(h)}_i\|_2\le L_h$.
   \end{lem}
 \begin{proof}
 With the notation of the partitioning \eqref{eq:partition} applied to $\varphi^{(h)}(z)$, from Lemma~\ref{lem:dario2} applied to $\varphi^{(h)}(z)$ we have
    $
      \sigma_j(C^{(h)}) \leq \lVert I-E^{(h)}(z) \rVert_2 \lVert S_{I-E^{(h)}}(z) \rVert_2 \sigma_j(\tilde C^{(h)})
    $.
     Thus, since $z\in\mathbb T$ and $I-E^{(h)}(z)$ is a submatrix of $\varphi^{(h)}(z)$,   we have  $\norm{I-E^{(h)}(z)}_2\leq \|\varphi^{(h)}(z)\|_2\le 1+3L_h$. Moreover, taking the norms in $S_{I-E^{(h)}}(z)=I-D^{(h)}(z)- C^{(h)}(z)(I-E^{(h)}(z))^{-1}B^{(h)}(z)$ we get 
$\|S_{I-E^{(h)}}(z)\|_2\le 1+L_h+L_h\|(I-E^{(h)}(z))^{-1}\|_2L_h$. 
Moreover,  for $z\in\mathbb T$ we have $|I-E^{(h)}(z))^{-1}|\le \sum_{i=0}^\infty E^{(h)}(1)^i$ so that $\|I-E^{(h)}(z))^{-1}\|_2\le \|(I-E^{(h)}(1))^{-1}\|_2=\|\sum_{i=0}^\infty E^{(h)}(1)^i\|_2 \le \|\sum_{i=0}^\infty A^{(h)}(1)^i\|_2=\|\varphi^{(h)}(1)^{-1}\|_2$, where we have set $A^{(h)}(z)=z^{-1}A_{-1}^{(h)}+A_0^{(h)}+zA_1^{(h)}$.
Here, we have used the property that the conditions $A_{-1}^{(h)}, A_0^{(h)},A_1^{(h)}\ge 0$ and  $\rho(A_{-1}^{(h)}+A_0^{(h)}+A_1^{(h)})<1$ are preserved at each step of CR (see \cite{SMC}).
Finally, since $\varphi^{(h)}(1)^{-1}=\psi^{(h)}(1)=\frac 1{N}\sum_{i=0}^{N-1}\psi(\omega_N^i)$, for $N=2^h$ (see Section \ref{sec:main-properties}), we have $\|\varphi^{(h)}(1)^{-1}\|_2\le \|\psi(1)\|_2$.
 \end{proof}
 \begin{rem}
 Note that the previous bound still holds with $\norm{\varphi^{(h)}(1)^{-1}}_2$, in place of $\norm{\varphi^{}(1)^{-1}}_2$. Experimentally,
 $\norm{\varphi^{(h)}(1)^{-1}}_2$
 is much smaller than $\norm{\varphi^{}(1)^{-1}}_2$
just after few steps $h$.
 \end{rem}
 Observe that $L_h$ depends on the step $h$ of CR. However, since under the assumptions of Theorem \ref{thm:decay}, the sequences generated by CR are such that $\lim_k A_i^{(h)}=0$, for $i=1,-1$ while $\lim _h A_0^{(h)}$ is finite (see \cite{SMC}), then there exists $L$ such that $L\ge L_h$.
 Thus,
 Combining Lemma~\ref{lem:well-conditioning} and Theorem~\ref{thm:decay} we obtain the following result.
 
\begin{cor}\label{cor:decay}
 Let $\varphi^{(h)}(z) = -z^{-1} A_{-1}^{(h)} + I - A_0^{(h)} - z A_1^{(h)}$ be the 
   $m\times m$-matrix Laurent polynomial obtained at the $h$th step of CR and assume the hypothesis of Theorem~\ref{thm:decay}.
    Then for any off-diagonal submatrix $C^{(h)}(z)$ of 
   $\varphi^{(h)}(z)$ we have
$
     \sigma_s(C^{(h)}) \leq 3M K \cdot e^{\frac{s-3k}{6k}\log t}
$,    
     where 
     $K=(1+3L)(1+L+L^2\norm{\varphi^{}(1)^{-1}}_2)$, $M$ is the constant defined in Theorem \ref{thm:decay} 
     and $L\ge \|A_i^{(h)}\|_2$, for $i=-1,0,1$. In particular, if $A_i$ is tridiagonal for $i=-1,0,1$ then $ \sigma_s(C^{(h)}) \leq M K \cdot e^{-(\frac s2 )\log t}$
\end{cor}

  \subsection{Exponential decay of the singular values in $A_i^{(h)}$}
  To prove the decay of the singular values in the off-diagonal submatrices of $A_i^{(h)}$ for $i=-1,0,1$ we rely 
on the following result of which we omit the elementary proof.
  
  \begin{lem} \label{lem:interpolation}
  Let
  $
    A(z) = z^{-1} A_{-1} + A_0 + z A_{1}
  $ and let $\xi$ be a
  primitive $6$-th root of the unity. 
  Then 
  $  A_{-1} = \frac{1}{3} \left( \xi A(\xi) + \xi^5 A(\xi^5) - A(-1) \right)$,
   $A_0 = \frac{1}{2} \left(A(z) + A(-z)\right)$,
   $A_1 = \frac{1}{3} \left( \xi^5 A(\xi) + \xi A(\xi^5) - A(-1) \right) $.
\end{lem}

\begin{lem} \label{lem:sumdecay}
  Let $A = \frac{1}{k} \sum_{i = 1}^k A_i \in \mathbb{C}^{n \times n}$ 
  where
  $
    \sigma_j(A_i) \leq \gamma e^{- \alpha j}$, for $j = 1, \ldots, n$.
  Then 
  $
    \sigma_j(A) \leq \tilde \gamma e^{-\alpha \frac{j - k}{k}  }, \quad 
    \tilde \gamma = \frac{\gamma}{1 - e^{-\alpha}}
  $.
\end{lem}

\begin{proof}
  Relying on the SVD, we write $A_i=\sum_{j=1}^{\infty}\sigma_j(A_i)u_{i,j}v_{i,j}^*$ where $u_{i,j}$ and $v_{i,j}$ are the singular vectors of $A_i$ and where, for convenience, we have expanded the sum to an infinite number of terms by setting $\sigma_j(A_i)=0$ for $j>n$.
  This allows us to write
  \[
    A = \frac{1}{k} \sum_{i = 1}^k A_i = 
    \sum_{j = 1}^{\infty} \left(
      \frac{1}{k} \sum_{i = 1}^k \sigma_j(A_i) u_{i,j} v_{i,j}^*
    \right) = \sum_{j = 1}^\infty \tilde A_j. 
  \]
  Observe that $\tilde A_j$ have rank
  $k$ and $\lVert A_j \rVert \leq \gamma e^{-\alpha j}$. Applying Lemma~\ref{dyads} completes the proof. 
\end{proof}

We may conclude with the decay property for the singular values of the off-diagonal submatrices of $A_i^{(h)}$, for $i=-1,0,1$.

\begin{lem}
  Let $ \varphi^{(h)}(z)$ be the matrix function 
  genertaed at
  the $h$th step of CR with the property
  that every offdiagonal submatrix $B(z)$ of $\varphi^{(h)}(z)$
  has decaying singular
  values such that $\sigma_s (B(z)) \leq \gamma e^{-\alpha s}$. Then every coefficient $B_i$ of $B(z) = z^{-1} B_{-1} + B_0 + 
  z B_{1}$ is such that
  $
    \sigma_s(B_0) \leq \gamma e^{- \alpha \frac{j-2}2 }$, $\sigma_s(B_i)\le
      \gamma e^{- \alpha  \frac{j-3}{3} }$,  for $ i=1,-1$.
\end{lem}

\begin{proof}
 By Lemma~\ref{lem:interpolation}, 
  we have an expression for $B_i$ based on evaluations
  of $B(z)$. In particular, we have
  $
    A_0 = \frac{1}{2} (B(i) + B(-i)), \quad 
    A_{\pm 1} = \frac{1}{3} (\xi^{\mp 1} B(\xi) + \xi^{\mp 5} B(\xi^5) - B(-1)),
  $
  where $\xi$ is a primitive $6$-th root of the unity. Applying Lemma~\ref{lem:sumdecay} completes the proof.
\end{proof}

  \subsection{The Markovian case}
  One of the most interesting application of CR algorithm is in the Markovian framework, in which applicability and convergence properties are guaranteed. In that case, the matrix function $\varphi$ satisfies almost all the hypotheses made in the previous subsections but it is singular at $z=1$ since $1$ is always an eigenvalue of $\varphi(z)$. Nevertheless we will show that Corollary~\ref{cor:decay} can  still be applied considering a rescaled version of $\varphi (z)$.
  
  When the coefficients $A_i$ for $i=-1,0,1$ represent the blocks of the transition  matrix of an irreducible not null recurrent QBD process, the eigenvalues of $\varphi(z)$ enjoy the following properties \cite{Gail,CR}:
  \begin{enumerate}[(i)]

  \item  $
|\lambda_1|\leq  |\lambda_2|\leq\dots\leq |\lambda_{m-1}|\leq \lambda_m<\lambda_{m+1}\leq |\lambda_{m+2}|\leq\dots\leq |\lambda_{2m}|
  $, with $\lambda_m,\lambda_{m+1}\in\mathbb R$ and one of the two equal to $1$.
  \item In the annulus $\{\lambda_m<|z|<\lambda_{m+1}\}$ $\varphi$ is invertible and the spectral radius of $I-\varphi(z)$ is strictly less than $1$.
  \end{enumerate}
  Hence we consider the rescaled version of $\varphi$, that is,
  $
\varphi_{\theta}(z):=\varphi(\theta z)  
  $,
  and we choose $\theta=\sqrt{\lambda_m \lambda_{m+1}}$. We obtain a matrix function invertible on  $\mathbb A=\{\frac{1}{t}<|z|<t\}$ where $t=\sqrt{\frac{\lambda_{m+1}}{\lambda_m}}$. 
  
  Observe that $\varphi_{\alpha}^{(h)}(z):=\varphi^{(h)}(\alpha^{2^h} z)$ so applying CR to  $\varphi_{\alpha}$ one obtains the same matrix sequences up to a rescaling factor. In particular the exponential  decay of the singular values is left unchanged as shown in the following.

    \begin{thm}\label{thm:uniform}
  For given $t>1$ and $\delta\ge 0$, consider the following class of matrix functions associated with QBD stochastic processes with $k$-quasiseparable blocks:
  \[
  \chi_{\delta,t}:=\left\{\varphi(z):\ \norm{\varphi^{-1}(z)}_2\leq \delta \quad 
t^{-1}\le |z|\le t,\quad t<\lambda_{m+1}/\lambda_m\right\}.
  \]
  Then there exists a uniform constant $\gamma(\delta,t)$ such that 
  for any off-diagonal block $C^{(h)}(z)$ of $\varphi^{(h)}(z)$, with  $\varphi\in\chi_{\delta,t}$, its $s$-th singular value is bounded by
  $
\sigma_s(C^{(h)}(z))\leq \gamma(\delta,t)\cdot e^{-\frac{s-3k}{6k}\log t}.  
  $
  \end{thm}
  
  \begin{rem}\rm
  Observe that in the case of null-recurrent QBD processes one has $\lambda_m=\lambda_{m+1}=1$, so that there is no open annulus including $\mathbb T$ where $\varphi(z)$ is nonsingular and  we cannot apply Theorem \ref{thm:decay}.
  This drawback can be partially overcome by applying the shift technique of \cite{SMC,CR}. This technique allows to construct a new matrix function $\tilde \varphi(z)$ which has the same eigenvalues of $\varphi(z)$ except for the eigenvalue 1 which is shifted to 0. So that $\tilde \varphi(z)$ has an open annulus containing $\mathbb T$ where it is nonsingular. Moreover, applying CR to $\tilde \varphi(z)$ generates matrix sequences which easily allow to recover the corresponding matrix sequences obtained by applying CR to $\varphi(z)$.
The sequences associated with $\varphi(z)$ differ from the sequences associated with $\tilde \varphi(z)$ by a rank-1 correction. This way, if the exponential decay of the singular values holds for the latter sequences, it holds also for the former ones.  
  The difficulty that still remains is that the nonnegativity of the blocks $A_{-1}$, $A_0$ and $A_1$ is not generally satisfied by the function $\tilde \varphi(z)$ so that in principle Theorem \ref{thm:decay} cannot be applied and a different version specific for this case should be formulated.
  \end{rem}
  
  \begin{rem}\rm
  The bounds that we have given to the decay of the singular values of the off-diagonal submatrices are not strict. Experimentally, singular values seem to decay slightly faster. Even in the null-recurrent case where $\lambda_m=\lambda_{m+1}=1$, the decay still occurs even though in a deteriorated form. The decay properties clearly depend on the domain of analyticity of $\psi(z)$ but this is not the only reason of the decay. More investigation is needed in this direction.
  \end{rem}

  \section{An algorithm using $\mathcal{H}$-matrices}
    \label{sec:hmatrices}
    We have provided an implementation of CR, which applies to matrix functions $\varphi(z)$ having quasiseparable blocks, and relies on the approximate quasiseparable structure induced by the decay of the singular values. We relied  on the $\mathcal H$-matrix representation of \cite{borm,borm2003,grasedyck2003}.

  \subsection{$\mathcal{H}$-matrix representation}
  Here, we give a brief and informal description of the $\mathcal{H}$-matrix representation that we have implemented. For full details we refer to \cite{borm} where an overview
  of the definition and use of hierarchical matrices is given.

Let $A\in\mathbb{R}^{n\times n}$ be a $k$-quasiseparable matrix such that
$A=\left[\begin{smallmatrix}A_{11}&A_{22}\\
A_{21}&A_{22}\end{smallmatrix}\right]$, 
$A_{11}\in\mathbb{R}^{n_1\times n_1}$, 
  $A_{22}\in\mathbb{R}^{n_2\times n_2}$,
  with $n_1:=\lfloor \frac{n}{2} \rfloor $ and $n_2:=\lceil \frac{n}{2} \rceil$. 
  Observe that the antidiagonal blocks $A_{12}$ and $A_{21}$ do not involve any element of the main diagonal of $A$, hence they are representable as a sum of at most $k$ dyads.
  Moreover the diagonal blocks $A_{11}$ and $A_{22}$ are square matrices with the same rank structure of $A$. Therefore these diagonal blocks are recursively represented with a similar partitioning. If blocks become small enough, they are stored as full matrices.
  
  \begin{figure}[!ht]
\centering
\includegraphics[width=0.1\textheight,height=0.1\textheight]{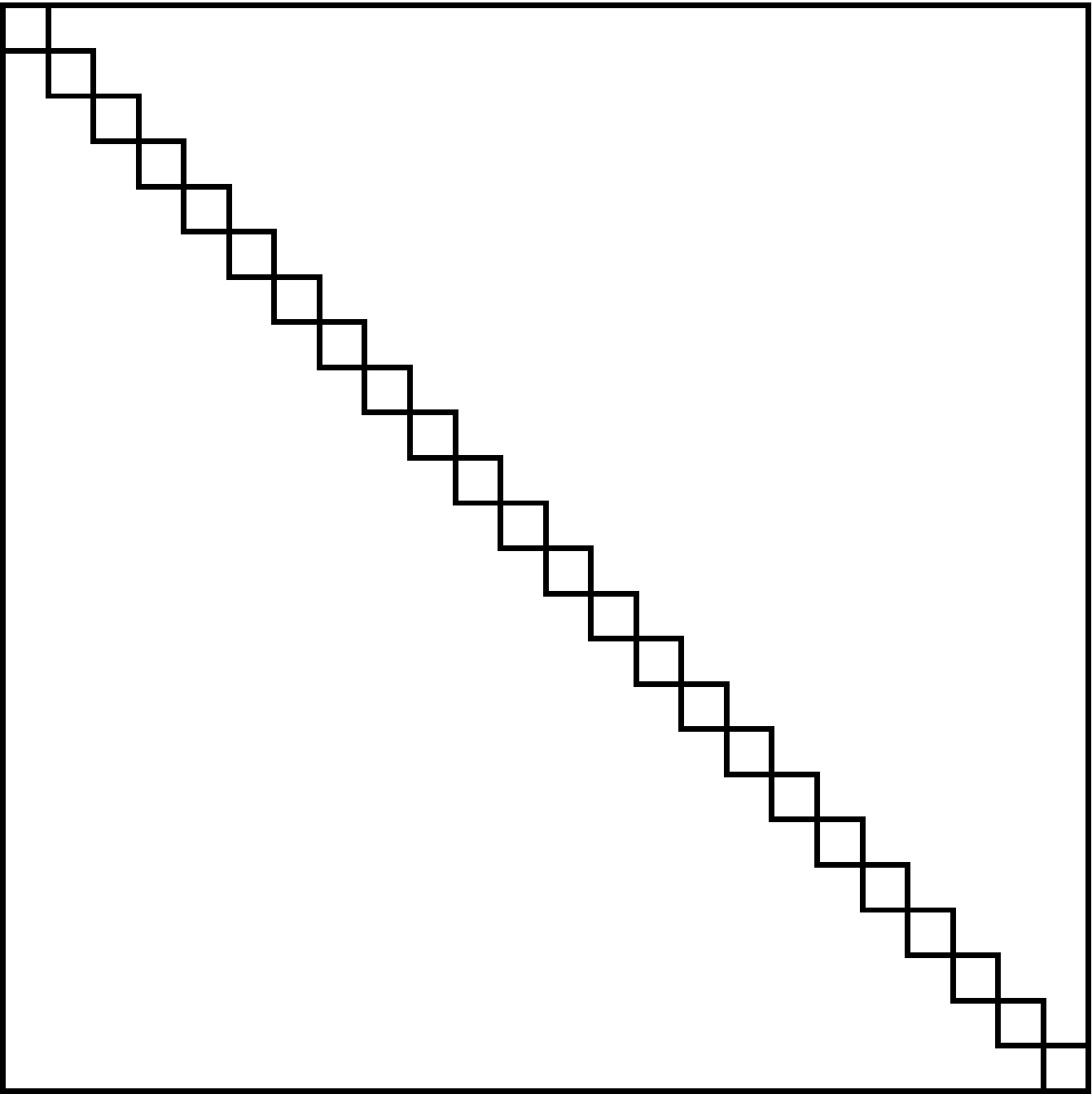}\qquad 
\includegraphics[width=0.1\textheight,height=0.1\textheight]{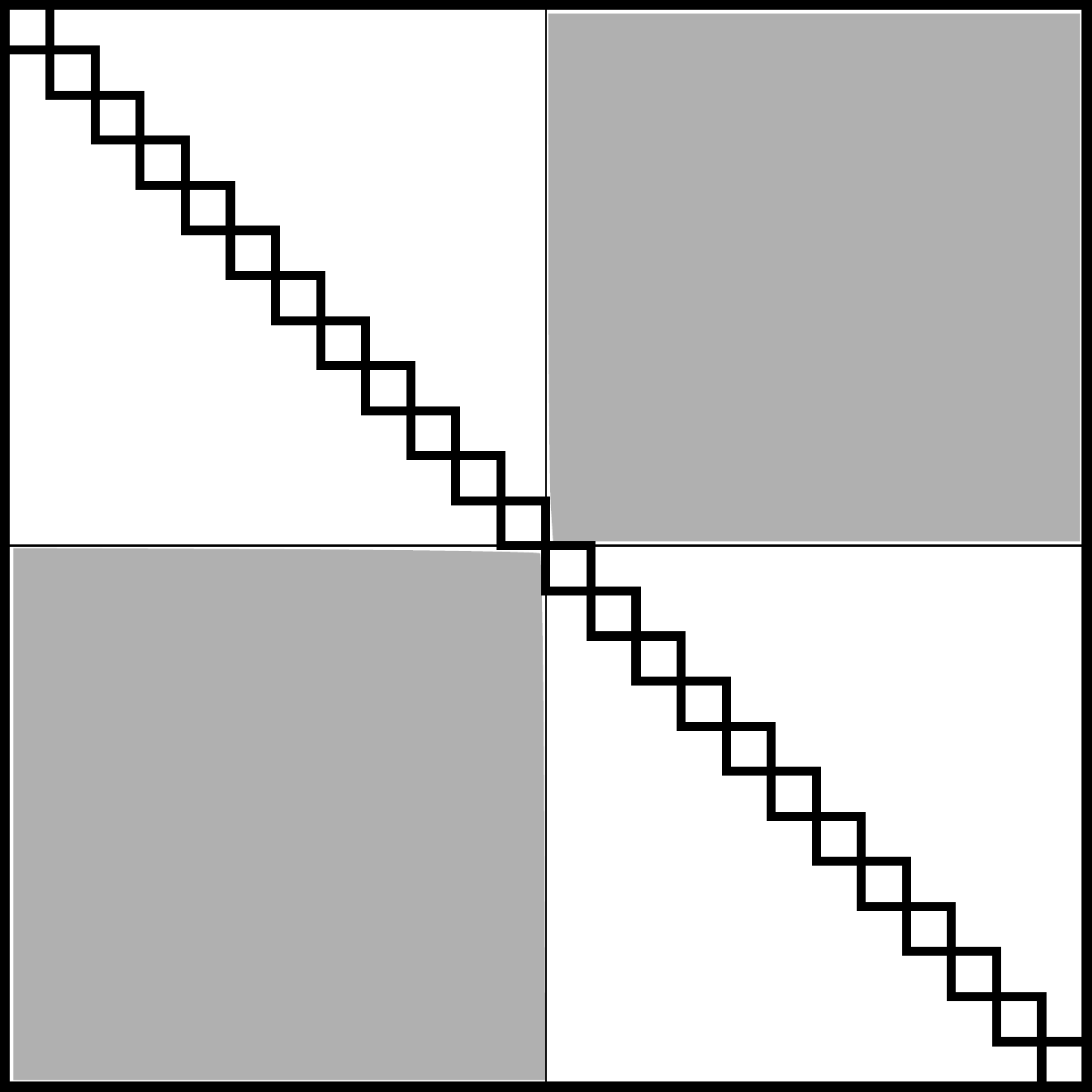}
\qquad 
\includegraphics[width=0.1\textheight,height=0.1\textheight]{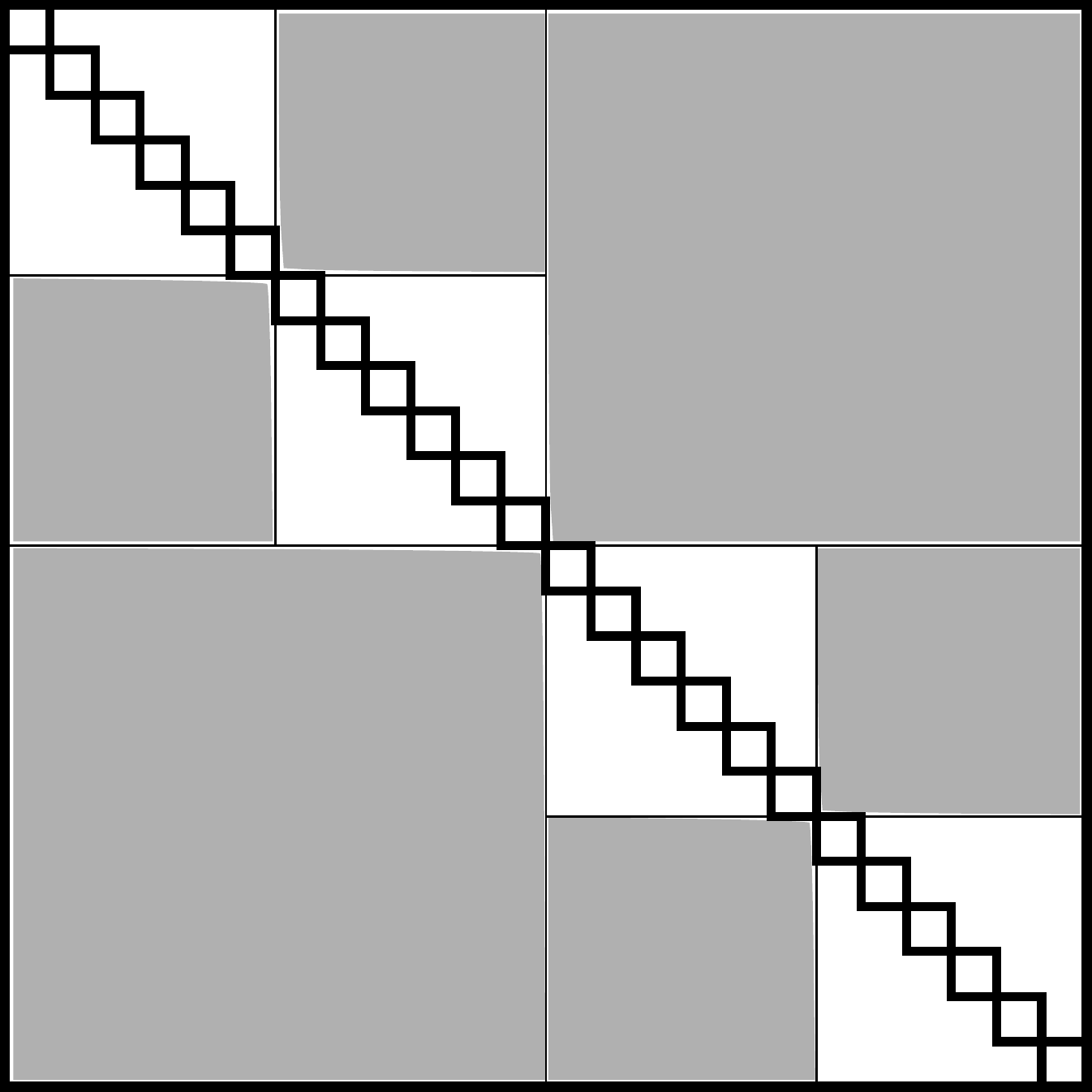}\qquad 
\includegraphics[width=0.1\textheight,height=0.1\textheight]{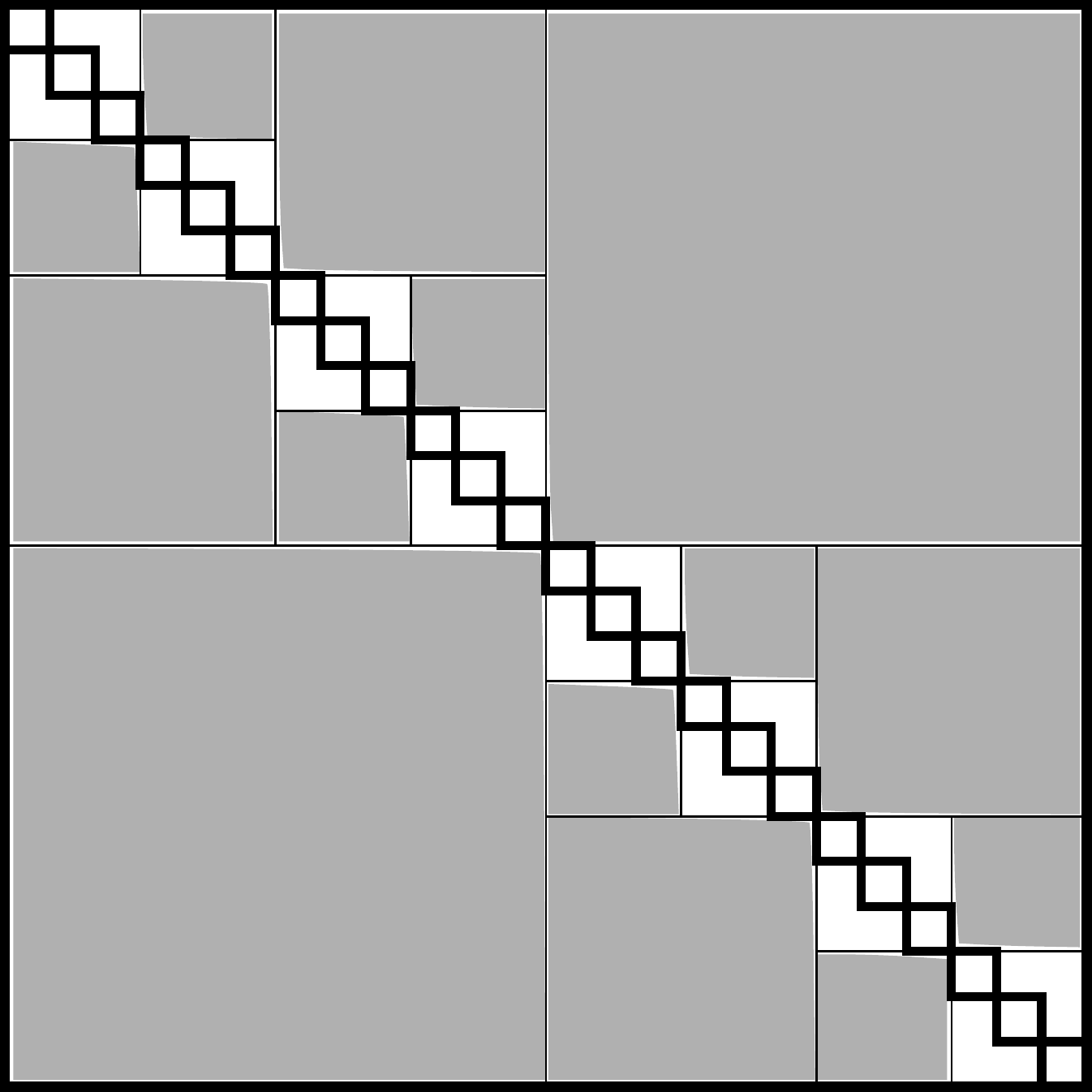}
 \caption{The behavior of the block partitioning in the $\mathcal{H}$-matrix representation. The blocks filled with grey are represented as sum of dyads, the diagonal blocks in the last step are stored as full matrices.}
\end{figure}

 \subsection{Quasiseparable CR}
 If the quasiseparable rank of the $\mathcal{H}$-matrices we are dealing with can be treated as a constant when compared to the dimension $n$, then the algorithms which perform the arithmetic operations have almost linear complexity \cite{borm}[Chapter 6].
 In particular we can achieve complexity $O(n\log n )$ for matrix addition and $O(n\log^2 n)$ for matrix multiplication and inversion. This is almost optimal, provided that the rank remains sufficiently low.
 
 In order to fully exploit the numerical quasiseparable structure we perform the arithmetic operations of CR adaptively with respect to the rank of the blocks. This means that the result of an arithmetic operation 
 (eg. matrix multiplication) will be an $\mathcal{H}$-matrix with the same partitioning, where each low rank block is a truncated reduced SVD of the corresponding block of the exact result. Hence
 the rank is not a priori fixed but depends on a threshold $\epsilon$ at which the truncation is done. The parameter $\epsilon$ can be regarded as the desired accuracy (for us is close to the machine precision $2.22\times 10^{-16}$)
 and can be crucial for the performance of the algorithm. 
 {\small
 \begin{figure}[!ht]
 \begin{center}
 \begin{tikzpicture}
  \begin{loglogaxis}[
     title=Execution time,
     xlabel=Size,
     ylabel=Time (s),
     legend entries={CR, $H_{10^{-16}}$, $H_{10^{-12}}$, $H_{10^{-8}}$},
     xmax=2e6,
     ymax = 1e5,
     width=.46\linewidth,
     height=.3\linewidth
   ]
   
   \addplot table {tempi_CR.dat};
   \addplot table {tempi_H16.dat};
   \addplot table {tempi_H12.dat};
   \addplot table {tempi_H8.dat};
  \end{loglogaxis}

 \end{tikzpicture}
 \begin{tikzpicture}
  \begin{loglogaxis}[
     title=Execution time ,
     xlabel=Band,
     ylabel=Time (s),
     legend entries={CR, $H_{10^{-16}}$, $H_{10^{-12}}$, $H_{10^{-8}}$},
     ymax=1e3,
     xmax=1e4,
     width=.46\linewidth,
     height=.3\linewidth
   ]
   
   \addplot table {tempi_CR_band.dat};
   \addplot table {tempi_H16_band.dat};
   \addplot table {tempi_H12_band.dat};
   \addplot table {tempi_H8_band.dat};
  \end{loglogaxis}

 \end{tikzpicture}
 \end{center}
 
 \caption{Timings of CR. To the left, CR is applied to tridiagonal blocks with increasing size. To right, CR is applied to band blocks with increasing band and size $1600$.} \label{fig:crtimings}
\end{figure}
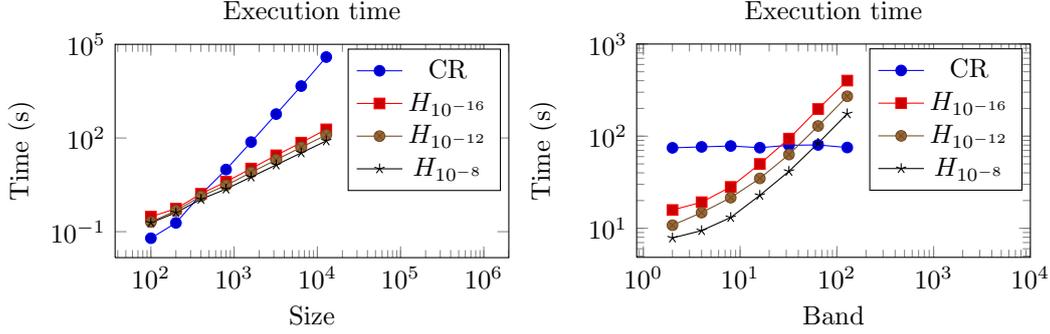
}

\begin{table}
\begin{center}
\rowcolors{1}{Gray}{white}
\resizebox{\textwidth}{!} {
 \footnotesize \begin{tabular}{c|cc|cc|cc|cc}
 \hiderowcolors
   & \multicolumn{2}{c|}{CR} & \multicolumn{2}{c|}{$H_{10^{-16}}$} & \multicolumn{2}{c|}{$H_{10^{-12}}$} & \multicolumn{2}{c}{$H_{10^{-8}}$} \\ \hline
   Size & Time (s) & Residue & Time (s) & Residue & Time (s) & Residue & Time (s) & Residue \\ \hline  
    \showrowcolors
   $100$ & $6.04e-02$ & $1.91e-16$ & $ 2.21e-01$ & $1.79e-15$ & $2.04e-01$ & $8.26e-14$ & $1.92e-01$ & $7.40e-10$ \\ 
$200$ & $1.88e-01$ & $2.51e-16$ & $ 5.78e-01$ & $1.39e-14$ & $5.03e-01$ & $1.01e-13$ & $4.29e-01$ & $2.29e-09$ \\ 
$400$ & $1.61e+01$ & $2.09e-16$ & $ 3.32e+00$ & $1.41e-14$ & $2.60e+00$ & $1.33e-13$ & $1.98e+00$ & $1.99e-09$ \\ 
$800$ & $2.63e+01$ & $2.74e-16$ & $ 4.55e+00$ & $1.94e-14$ & $3.49e+00$ & $2.71e-13$ & $2.63e+00$ & $2.69e-09$ \\ 
$1600$ & $8.12e+01$ & $3.82e-12$ & $ 1.18e+01$ & $3.82e-12$ & $8.78e+00$ & $3.82e-12$ & $6.24e+00$ & $3.39e-09$ \\ 
$3200$ & $6.35e+02$ & $5.46e-08$ & $ 3.12e+01$ & $5.46e-08$ & $2.21e+01$ & $5.46e-08$ & $1.51e+01$ & $5.43e-08$ \\ 
$6400$ & $5.03e+03$ & $3.89e-08$ & $ 7.83e+01$ & $3.89e-08$ & $5.38e+01$ & $3.89e-08$ & $3.58e+01$ & $3.87e-08$ \\ 
$12800$ & $4.06e+04$ & $1.99e-08$ & $ 1.94e+02$ & $1.99e-08$ & $1.29e+02$ & $1.99e-08$ & $8.37e+01$ & $1.97e-08$ \\
 \end{tabular}
}
\end{center}
\caption{Timings and accuracy for $15$ iterations of CR at the increasing of the size of the blocks.}\label{tab:num}
\end{table}
\begin{table}
\begin{center}
\rowcolors{1}{Gray}{white}
\resizebox{\textwidth}{!} {
 \footnotesize \begin{tabular}{c|cc|cc|cc|cc}
 \hiderowcolors
   & \multicolumn{2}{c|}{CR} & \multicolumn{2}{c|}{$H_{10^{-16}}$} & \multicolumn{2}{c|}{$H_{10^{-12}}$} & \multicolumn{2}{c}{$H_{10^{-8}}$} \\ \hline
   Band & Time (s) & Residue & Time (s) & Residue & Time (s) & Residue & Time (s) & Residue \\ \hline  
    \showrowcolors
2 &$ 7.47e+01 $&$ 2.11e-16 $&$  1.58e+01 $&$ 6.95e-15 $&$ 1.08e+01 $&$ 2.62e-13 $&$ 7.86e+00 $&$ 2.57e-09 $\\ 
4 &$ 7.65e+01 $&$ 1.66e-16 $&$  1.92e+01 $&$ 4.88e-15 $&$ 1.48e+01 $&$ 2.36e-13 $&$ 9.44e+00 $&$ 3.15e-09 $\\ 
8 &$ 7.82e+01 $&$ 1.48e-16 $&$  2.81e+01 $&$ 6.11e-15 $&$ 2.15e+01 $&$ 2.08e-13 $&$ 1.31e+01 $&$ 2.10e-09 $\\ 
16 &$ 7.50e+01 $&$ 1.35e-16 $&$  4.99e+01 $&$ 4.98e-15 $&$ 3.48e+01 $&$ 2.29e-13 $&$ 2.28e+01 $&$ 2.08e-09 $\\ 
32 &$ 7.97e+01 $&$ 1.33e-16 $&$  9.40e+01 $&$ 5.79e-15 $&$ 6.32e+01 $&$ 2.01e-13 $&$ 4.15e+01 $&$ 2.28e-09 $\\ 
64 &$ 8.03e+01 $&$ 1.31e-16 $&$  1.97e+02 $&$ 6.79e-15 $&$ 1.29e+02 $&$ 1.99e-13 $&$ 8.37e+01 $&$ 2.01e-09 $\\ 
128 &$ 7.53e+01 $&$ 1.28e-16 $&$  4.01e+02 $&$ 5.89e-15 $&$ 2.71e+02 $&$ 2.02e-13 $&$ 1.75e+02 $&$ 2.15e-09$\\
 \end{tabular}
}
\end{center}
\caption{Timings and accuracy for $15$ iterations of CR on blocks with size $1600$ with different bands.}
\end{table}
\subsection{A note on the implementation}

For the implementation of this algorithm we relied on
the open source library {\tt H2Lib} \cite{h2lib}. The library
has been wrapped in MEX files for use in MATLAB, where the 
numerical experiments have been run. The code 
developed in this context is freely available at \cite{h2lib-matlab}. The bindings developed in the
testing of the algorithm are only a partial mapping 
of all the routines available in the original
{\tt H2Lib} library but we feel that it is worth making
them public so they can be used as a base for a further
extension. 

For a fair comparison, we have compiled {\tt H2Lib} with the LAPACK library used by MATLAB. Moreover, we have disabled the
parallelism in the Intel MKL library to obtain more accurate
results. It is important to notice that running with
parallelism enabled in the MKL library leads to improved
performance both for {\tt H2LIb} and for MATLAB, but the
improvement is more relevant in the latter. This is due to
the fact that the library is optimized for the multiplication
of large matrices, such as in the full CR implementation (when full matrices of large size are multiplied together). The multiplication of the small rectangular matrices involved in
the hierarchical representation, instead, benefit less from
this implementation. Anyway, also in this case we see that
our implementation is more efficient even if starting 
from larger dimension. For example, on a Xeon 
server with 24 threads available our implementation is faster than the standard one
approximately for $n > 500$. 

Table \ref{tab:num} reports the results of some numerical experiments, where in each column we have reported: the size of the blocks from $m=100$ up to $m=12800$, the CPU time, in seconds, required by standard CR and the residual error, then from column 3 to column 5 we reported the CPU time, in seconds, and the residual error of our implementation with values of $\epsilon=10^{-16},\, 10^{-12},\, 10^{-8}$, respectively.  It is interesting to observe that the precision of the result does not deteriorate much for large values of $m$. Moreover, the speed-up that we get goes beyond two order of magnitude.

 In Table $2$ we repeat the experiment fixing the size to $1600$ and letting the band of the starting blocks to increase exponentially from $2$ up to $128$. It should be note that the gain of time of our implementation seems to deteriorate linearly with respect to the increase of the band.

In Figure \ref{fig:crtimings} we give a graphic description, in logarithmic scale, of the growth of the CPU time in the latter experiments. The test problems are generated randomly.

 \section{Conclusions}
   \label{sec:conclusions}
We have experimentally observed the exponential decay of the singular values of certain off-diagonal submatrices generated by cyclic reduction applied to certain QBD stochastic processes of practical interest. We have formally related this property to the width of the domain of analyticity of the inverse matrix function associated with the QBD. 

We have provided a software implementation of CR, for QBD with tridiagonal blocks encountered in the analysis of bidimensional random walk, which relies on this decay property. The speed up that we get with respect to standard CR is substantial even with moderately large size of the blocks.
  
Even though experiments confirm the validity of the  theoretical bounds, the bounds obtained in our analysis are not sharp with respect to the values actually encountered in our computational experiments. This shows that the decay property depends also on other factors which deserve further investigation. This is the aim of our future research.

\bibliographystyle{abbrv}
\bibliography{bibliography}

\end{document}